\documentclass[oneside,11pt,a4paper]{amsart}

\usepackage{amsmath,yhmath}
\usepackage{amssymb}
\usepackage{amsfonts}
\usepackage{bbm} 
\usepackage{mathtools}
\usepackage{tikz}
\usepackage{tikz-cd}
\usepackage[inline]{enumitem}
\usepackage[mathscr]{eucal} 
\allowdisplaybreaks
\usepackage{stmaryrd} 

\numberwithin{equation}{section}
\usepackage{caption}
\usepackage{geometry}
\usepackage{tikz}
\tikzset{Equal/.style={-,double line with arrow={-,-}}} 
\usetikzlibrary{decorations.markings}
\tikzset{double line with arrow/.style args={#1,#2}{decorate,decoration={markings,
mark=at position 0 with {\coordinate (ta-base-1) at (0,1pt);
\coordinate (ta-base-2) at (0,-1pt);},
mark=at position 1 with {\draw[#1] (ta-base-1) -- (0,1pt);
\draw[#2] (ta-base-2) -- (0,-1pt);
}}}}

\usepackage{aliascnt}
\usepackage[colorlinks, linkcolor=black,citecolor=blue, urlcolor=cyan, unicode, bookmarksnumbered]{hyperref}

\newcommand*{\eeqref}[2][Equation~]{%
  \hyperref[{#2}]{#1(\ref*{#2})}%
}

\newtheoremstyle{note}
{\topsep}
{\topsep}
{}
{0pt}
{\bfseries}
{.}
{0.5em }
{}
\theoremstyle{plain}

\newtheorem{Thm}{Theorem}[section]

\newtheorem*{Thm*}{Theorem}

\newaliascnt{prop}{Thm}
\newtheorem{Prop}[prop]{Proposition}
\aliascntresetthe{prop}

\newaliascnt{lemma}{Thm}
\newtheorem{Lemma}[lemma]{Lemma}
\aliascntresetthe{lemma}

\newaliascnt{coro}{Thm}
\newtheorem{Coro}[coro]{Corollary}
\aliascntresetthe{coro}

\newaliascnt{conjecture}{Thm}

\aliascntresetthe{conjecture}

\newtheorem{Lemma*}{Lemma}
\newtheorem*{Prop*}{Proposition}
\newtheorem*{Coro*}{Corollary}

\theoremstyle{definition}

\newaliascnt{def}{Thm}
\newtheorem{Def}[def]{Definition}
\aliascntresetthe{def}

\newtheorem*{Def*}{Definition}

\newaliascnt{Eg}{Thm}

\aliascntresetthe{Eg}

\newtheorem*{eg*}{Example}

\theoremstyle{remark}

\newaliascnt{rmk}{Thm}
\newtheorem{RMK}[rmk]{Remark}
\aliascntresetthe{rmk}

\newtheorem*{RMK*}{Remark}

\theoremstyle{plain}

\newtheorem*{iThm*}{Theorem}

\newaliascnt{iprop}{iThm}

\aliascntresetthe{iprop}

\newaliascnt{iCoro}{iThm}

\aliascntresetthe{iCoro}

\newaliascnt{iconjecture}{iThm}

\aliascntresetthe{iconjecture}

\theoremstyle{definition}

\newaliascnt{idef}{iThm}

\aliascntresetthe{idef}

\newtheorem*{iDef*}{Definition}

\newaliascnt{iEg}{iThm}

\aliascntresetthe{iEg}

\newtheorem*{ieg*}{Example}

\theoremstyle{remark}

\newaliascnt{irmk}{iThm}

\aliascntresetthe{irmk}

\newtheorem*{iRMK*}{Remark}

\newcommand{\bZ} { {\mathbb{Z}}}
\newcommand{\bR} { {\mathbb{R}}}
\newcommand{\bK} { {\mathbb{K}}}
\newcommand{\bS} { {\mathbb{S}}}
\newcommand{\bT} { {\mathbb{T}}}

\newcommand{\bfk} { {\mathbf{k}}}
\newcommand{\cA} { {\mathcal{A}}}
\newcommand{\cB} { {\mathcal{B}}}
\newcommand{\cC} { {\mathcal{C}}}

\newcommand{\cE} { {{\mathcal{E}}}}

\newcommand{\sT}{ \mathscr{T} }
\newcommand{\supp}{{\textnormal{supp}}}
\renewcommand{\SS}{{ \operatorname{SS}  }}

\newcommand{\HOM}{{ \textnormal{Hom}  }}
\newcommand{\HHOM}{{ \mathcal{H}om  }}

\newcommand{\Mod}{{\textnormal{Mod}}}

\newcommand{\QCoha}{{\operatorname{aQCoh}}}

\newcommand{\Sh}{{\operatorname{Sh}}}
\newcommand{\Fun}{{\operatorname{Fun}}}

\newcommand{\Sp}{{\operatorname{Sp}}}
\newcommand{\Catdual}{{\operatorname{Cat^{dual}_{st}}}}
\newcommand{\Motloc}{{\operatorname{Mot_{loc}}}}
\newcommand{\Uloc}{{\mathcal{U}^{\textnormal{cont}}_{\textnormal{loc}}}}
\newcommand{\FST}{{\operatorname{FST}}}
\newcommand{\Nov}{{\operatorname{Nov}}}

\title{A remark on Continuous K-theory and Fourier-Sato transform}
\author{Bingyu Zhang}
\date{}

\begin{document}

\begin{abstract}
In this note, we prove a generalization of Efimov's computation for the universal localizing invariant of categories of sheaves with certain microsupport constraints. The proof is based on certain categorical equivalences given by the Fourier-Sato transform, which is different from the original proof. As an application, we compute the universal localizing invariant of the category of almost quasi-coherent sheaves on the Novikov toric scheme introduced by Vaintrob.
\end{abstract}

\maketitle

\section{Introduction}
In \cite{Efimov-K-theory}, Efimov introduces an algebraic K-theory for a class of large categories, namely dualizable stable categories, which extends the usual non-connective algebraic K-theory defined for certain small categories. In general, the construction enables us to extend localizing invariants from small categories to dualizable stable categories. In particular, for the universal (finitary) localizing invariant $\mathcal{U}_{loc}$, there exists a canonical extension $\Uloc: \Catdual \rightarrow \Motloc$, where $\Catdual$ is the category of dualizable stable categories (and strongly continuous functors between them) and $\Motloc$ is the category of non-commutative motives introduced by \cite{universal-hihger-K}. 

However, as the construction involves computations about the so-called Calkin category that is not easy to describe, computation of the continuous version of localizing invariants is even harder, and few computational results are known. One distinguished result among them is the following:
\begin{Thm}[{{\cite[Theorem 6.11]{Efimov-K-theory}}}]\label{Theorem: Efimov no support} Let $X$ be a locally compact Hausdorff space and $\underline{\mathcal{C}}$ be a presheaf on $X$ with values in $\Catdual$. Then the category $\Sh(X;\underline{\mathcal{C}})$ is dualizable stable and we have the following natural isomorphism in $\Motloc$:
\[\Uloc (\Sh(X;\underline{\mathcal{C}})) \simeq \Gamma_c(X, (\Uloc\underline{\mathcal{C}})^\sharp).\]   
\end{Thm}

Another interesting example concerns categories of sheaves with microsupport constraints. For a manifold $M$ and any $F\in \Sh(M)$, Kashiwara and Schapira introduced a conic closed set $\SS(F)\subset T^*M$ in \cite{KS90}, which is called the microsupport of sheaves. For a conic closed set $Z\subset T^*M$, we denote by $\Sh_{Z}(M;{\mathcal{C}})$ the full subcategory of sheaves whose microsupport is bounded by $Z$. Then $\Sh_{Z}(M;{\mathcal{C}})$ is dualizable stable (when $\cC$ is) since it is a reflexive subcategory of $\Sh(M;{\mathcal{C}})$.

\begin{Thm}[{\cite[Proposition 4.21]{Efimov-K-theory}}]\label{Theorem: Efimov with support}Let ${\mathcal{C}}$ be a dualizable stable category. Then for the category $\Sh_{\bR\times [0,\infty)}(\bR;{\mathcal{C}})$, which is known to be dualizable stable, we have the following natural equivalence in $\Motloc$:
\[\Uloc (\Sh_{\bR\times [0,\infty)}(\bR;{\mathcal{C}})) \simeq 0.\] 
\end{Thm}

It was explained by Alexander I. Efimov, during the \textit{Masterclass: Continuous K-theory} in University of Copenhagen on June 2024, that \autoref{Theorem: Efimov with support} is still true for a finite dimensional real vector space $V$ with the microsupport constraint $Z=V\times \gamma$ for a non-zero proper closed convex cone $\gamma$ (i.e. Equation \eqref{equation: HD cone} below). We are grateful for his generosity. One can prove the high dimensional version in the same way as the 1-dimensional version using a $V$-indexed semi-orthogonal decomposition.

\subsection*{New results}

In this article, using the Fourier-Sato transform, we directly identify certain categories of sheaves with microsupport constraints with certain categories without microsupport constraints. Those facts are well known to experts; however, the interesting part is that we can use them to deduce a generalization of \autoref{Theorem: Efimov with support} directly from \autoref{Theorem: Efimov no support}.

To achieve our target, we explain a definition of microsupport for more general coefficient categories as described in \cite[Remark 4.23]{Efimov-K-theory}. In particular, we show that if $\cC$ is presentable stable, then the definition inherits most of the nice properties deduced in \cite{KS90}. In particular, we may develop a $\Motloc$-valued microlocal sheaf theory without any difficulty. Those constructions may be of independent interest.

Our main result is
\begin{Thm}[{\autoref{Theorem: main thm} (2-b)} below]\label{Theorem: conic main thm}Let ${\mathcal{C}}$ be a dualizable stable category. For a finite dimensional real vector space $V$ and a conic closed set $X\subset V^\vee$, we have the equivalence 
\[\Uloc (\Sh_{V\times X}(V;\cC)) \simeq \Gamma_c(X; \Uloc({\cC}) ).\]
\end{Thm}
In particular, if we pick $X=\gamma$, a non-zero proper convex closed cone. Then we deduce from a direct cohomology computation that
\begin{equation}\label{equation: HD cone}
    \Uloc (\Sh_{V\times \gamma}(V;{\cC})) \simeq \Gamma_c(\,\gamma\,; \Uloc({\cC}) )=0,
\end{equation}
which is the straightforward generalization of \autoref{Theorem: Efimov with support}.

\begin{RMK}The proof of \autoref{Theorem: Efimov with support} therein is better in the sense that it could be generalized to all accessible localizing invariants, in particular, this also works for Equation \eqref{equation: HD cone}. However, so far we only know \autoref{Theorem: conic main thm} works for finitary localizing invariants due to the state of \autoref{Theorem: Efimov no support}.    
\end{RMK}

\begin{RMK} Here, we explain the logic dependence of results. 

The original proof of \autoref{Theorem: Efimov with support} (and Equation \eqref{equation: HD cone}) uses the microlocal cut-off lemma of Kashiwara and Schapira to identify the corresponding categories to sheaves over the so-called $\gamma^\vee$-topology (which is non-Hausdorff!). Then it is concluded by a semi-orthogonal decomposition of corresponding presheaf categories. No other machinery in microlocal sheaf theory is involved.

Our proof utilizes more machinery from microlocal sheaf theory; unsurprisingly, the microlocal cut-off lemma appears implicitly in our approach. However, we will not construct any semi-orthogonal decomposition, which makes our proof different from the original one.    
\end{RMK}

We also prove a version of the theorem for the Tamarkin category $\sT_{V\times X}(T^*V)\subset \Sh(V;\sT_\cC)$ that will be defined later. 
\begin{Thm}[{\autoref{Theorem: main thm} (1)} below]\label{Theorem: main thm intro}Let ${\mathcal{C}}$ be a dualizable stable category. For a finite dimensional real vector space $V$ and a closed set $X\subset V^\vee$, we have the equivalence 
\[\Uloc (\sT_{V\times X}(T^*V)) \simeq \Omega\Gamma_c(X; \Uloc({\cC}) ).\]
\end{Thm}

\begin{RMK}Here, I want to emphasize that the coefficient category $\sT_\cC$ here has wide range of applications. The category $\sT_\cC$ was first introduced by Tamarkin \cite{tamarkin2013} for applications in symplectic geometry. It was also applied to resolve the irregular Riemann-Hilbert correspondence (with the name enhanced sheaves) \cite{DAgnoloKashiwara}. Recently, $\sT_\cC$ is considered in \cite{WildBetti} in terms of the notation $\mathbb{W}$ with a slightly different definition. Based on an idea of Vaintrob, we have $\mathbb{W}\simeq \sT_\Sp$, which is also equivalent to the category of complete almost ($\bR$-filtered) modules over the Novikov ring, see \cite[Proposition 4.12]{APT-Tatsuki-Zhang} for more detail. As explained in \cite{WildBetti}, the category $\sT_\cC$ would have interesting applications in analytic geometry since the Novikov ring could be thought of as a ``perfectoid ring over $\bZ$".    
\end{RMK}

\subsection*{Application}
Next, we present an application. For a fan\footnote{We do not ask the fan to be rational with respect to a fixed lattice in $\bR^n$.} $\Sigma$ in $\bR^n$, Vaintrob constructs a non-Noetherian $\bfk$-scheme (where $\bfk$ is a discrete ring), the so-called Novikov toric scheme, $X_\Sigma^\Nov$ and a subscheme $\partial_\Sigma$ defined by an idempotent ideal sheaf in \cite{Vaintrob-logCCC}. Then we can discuss the category of almost coherent sheaves on the almost content $(X_\Sigma^\Nov,\partial_\Sigma)$. If $\Sigma$ is rational, $X_\Sigma^\Nov$ is strongly related to the infinite root stack $\sqrt[\infty]{X_\Sigma}$ of the usual toric variety. We refer to \cite{APT-Tatsuki-Zhang} for more details.

We have the following result, which is first proven by Vaintrob and then by Kuwagaki and the author \cite{APT-Tatsuki-Zhang} using a different method. 
\begin{Thm}For a fan $\Sigma$ and $\Mod_\bfk$ the category of $\bfk$-modules, we have
\[\QCoha_{\bT^\Nov}(X_\Sigma^\Nov,\partial_\Sigma) \simeq \Sh_{\bR^n\times |\Sigma|} (\bR^n;\Mod_\bfk).\]    
\end{Thm}
Then as an application of \autoref{Theorem: conic main thm}, we have
\begin{Coro}For a fan $\Sigma$, we have
\[\Uloc (\QCoha_{\bT^\Nov}(X_\Sigma^\Nov,\partial_\Sigma))\simeq \Gamma_c(|\Sigma|; \Uloc(\Mod_\bfk) ).\]    
\end{Coro}

\begin{RMK}Instead of \autoref{Theorem: conic main thm}, one can also deduce the Corollary from \eqref{equation: HD cone} and \cite[Proposition 4.11]{Efimov-K-theory} based on the fiber product decomposition of $\QCoha_{\bT^\Nov}(X_\Sigma^\Nov,\partial_\Sigma)$ explained in \cite{APT-Tatsuki-Zhang}.
    
\end{RMK}

\subsubsection*{Category convention}
In this article, we always mean $\infty$-categories when referring to categories. We denote by $\Catdual$ the category of dualizable stable categories consisting of presentable stable categories that are dualizable with respect to the Lurie tensor product and strongly continuous functors between them. In particular, compactly generated stable categories are dualizable. We denote $\Sp$ as the category of spectra, and its unit, the spherical spectrum, is denoted by $\bS$. We denote the countable cardinality by $\omega$.

\subsection*{Acknowledgements} The author thanks Alexander I. Efimov and Peter Scholze for helpful discussions. This work was supported by the Novo Nordisk Foundation grant NNF20OC0066298 and the VILLUM FONDEN Investigator grant 37814.

\section{Sheaves and microsupport}
For a bi-complete stable category $\cC$ and a topological space $X$, we denote the category of $\cC$-valued sheaves by $\Sh(X;\cC)$ \cite[7.3.3.1]{HTT}. It is explained in \cite{6functor-infinity} that in this case, we have an equivalence $\Sh(X;\Sp)\otimes \cC \simeq \Sh(X;\cC)$.

It is further explained in loc. cit. that in the $\cC$-valued sheaf setting, we can define the functors $f^\cC_*, f^\cC_!$, and $f^*_\cC, f^!_\cC$. When $\cC$ is, in addition, symmetric monoidal, we can define the monoidal product $\otimes_\cC$ and the internal hom $\HHOM_\cC$ as the right adjoint of $\otimes_\cC$, yielding the full six-functor formalism. We also refer to \cite{Six-FunctorScholze} for further details on the six-functor formalism.

Regarding the microlocal theory of sheaves in the $\infty$-categorical setting, we remark that, based on \cite{Amicrolocallemma_infinitycat}, all arguments in \cite{KS90} extend to the case where $\cC$ is compactly generated. Therefore, the microlocal sheaf theory with compactly generated stable coefficients, for example, $\cC = \Sp$, can be developed without difficulty.

However, as suggested in \cite[Remark 4.24]{Efimov-K-theory}, it is possible to develop a microlocal sheaf theory with more general coefficients, where the original idea comes from \cite{Beilinson_epsilonfactor}.

Here, we present the $\Omega$-lens definition. Let $M$ be a smooth manifold, and set $\dot{T}^*M = T^*M \setminus 0_M$.

\begin{Def}{\cite[Definition 3.1]{guillermouviterbo_gammasupport}} \label{Def: Omega-lens}
Let $\Omega \subset \dot{T}^*M$ be an open conic subset. We call a locally closed subset $C(g)$ of $M$ an \emph{$\Omega$-lens} if the following conditions are satisfied: $\overline{C(g)}$ is compact, and there exists an open neighborhood $U$ of $\overline{C(g)}$ and a $C^\infty$ function $g\colon U \times [0,1] \to \bR$ such that
\begin{enumerate}
\item $dg_t(x) \in \Omega$ for all $(x,t) \in U \times [0,1]$, where $g_t = g|_{U\times\{t\}}$;
\item $\{g_t<0\} \subset \{g_{t'}<0\}$ if $t \leq t'$;
\item the hypersurfaces $\{g_t = 0\}$ coincide on $U \setminus \overline{C(g)}$;
\item $C(g) = \{g_1<0\} \setminus \{g_0<0\}$.
\end{enumerate}
\end{Def}

\begin{Def}\label{def: microsupport}
Let $\cC$ be a category that admits small limits, and let $F \in \Sh(M;\cC)$. We say $F$ is regular on $\Omega$ if for any $\Omega$-lens $C(g)$ defined by a smooth function $g$, the restriction morphism
\[
\Gamma(\{g_1<0\}, F) \rightarrow \Gamma(\{g_0<0\}, F)
\]
is an equivalence.

For a conic closed subset $Z\subset T^*M$, we set $\Sh_Z(M;\cC)$ as the full subcategory of $\Sh(M;\cC)$ spanned by $F$ that is regular on $\dot{T}^*M\setminus Z$. 

We define $\dot{T}^*M \setminus \dot{\SS}_\cC(F)$ as the maximal open subset (if it exists) $\Omega \subset \dot{T}^*M$ such that $F$ is regular on $\dot{T}^*M \setminus \dot{\SS}_\cC(F)$. We then define the microsupport of $F$ by $\SS_\cC(F) = \dot{\SS}_\cC(F) \cup \supp(F)$. When the category $\cC$ is clear from context, we simply write $\SS(F)$.
\end{Def}

\begin{RMK}
    In \autoref{Def: Omega-lens}, $g$ relies on open neighborhood $U$ of $\overline{C(g)}$. But in \autoref{def: microsupport}, regularity is independent of $U$ and is only dependent on $\{g_i<0\}$ for $i=0,1$.
\end{RMK}

When $\cC = \Sp$, there exists a pointwise definition of microsupport $\SS_{KS}(F) \subset T^*M$, as explained in \cite{KS90,Amicrolocallemma_infinitycat}. In particular, $\SS_{KS}(F)$ exists for all $F$. The argument therein shows that $\SS_{KS}(F) \cap 0_M = \supp(F)$, and hence $\SS_{KS}(F) \cap 0_M = \SS_{\Sp}(F) \cap 0_M$. For the non-zero part $\dot{\SS}_{KS}(F)=\SS_{KS}(F)\setminus 0_M$, we have the following:

\begin{Lemma}{\cite[Lemma 3.2]{guillermouviterbo_gammasupport}}\label{lemma: Omega-lens1}
Let $F \in \Sh(M;\Sp)$ and let $\Omega \subset T^*M \setminus 0_M$ be an open conic subset. Then $\dot{\SS}_{KS}(F) \cap \Omega = \emptyset$ if and only if $\HOM(\bS_{C(g)}, F) \simeq 0$ for any $\Omega$-lens $C(g)$ (i.e. $F$ is regular on $\Omega$).
\end{Lemma}

\begin{Coro}
For $F \in \Sh(M;\Sp)$, we have $\SS_{KS}(F) = \SS_{\Sp}(F)$. Consequently, $\SS_{\Sp}(F)$ exists.
\end{Coro}

On the other hand, there is a version of the covariant Verdier duality with microsupport: 
\begin{Prop}{\cite[Theorem B.8]{APT-Tatsuki-Zhang}}\label{lemma: Omega-lens2}For a stable category $\cC$ that admits both small limits and colimits, we have $F$ is regular on $\Omega$ if and only if $\Gamma_c( \{g_0<0\},F) \rightarrow \Gamma_c( \{g_1<0\},F)$ is an equivalence for any $-\Omega$-lens $C(g)$.
\end{Prop}

The \autoref{def: microsupport} shows that $\Sh_Z(M;\cC)$ is closed under limits, and \autoref{lemma: Omega-lens2} shows that $\Sh_Z(M;\cC)$ is closed under colimits when $\cC$ is stable. In particular, when $\cC$ is dualizable stable, we have $\Sh_Z(M;\cC)$ is dualizable stable.

For any presentable stable category $\cC$, under the natural identification $\Sh(M;\Sp)\otimes \cC \simeq \Sh(M;\cC)$, we have $i_Z^\Sp\otimes\cC: \Sh_Z(M;\Sp)\otimes \cC \rightarrow \Sh(M;\cC)$ is fully-faithful and admits both left and right adjoints. Moreover, we have

\begin{Prop}[{{\cite[Remark 4.24]{Efimov-K-theory}}}]\label{defintion: ad-hoc def of SS}For a presentable stable category $\cC$ and a conic closed subset $Z\subset T^*M$, the essential image of the functor $\Sh_Z(M;\Sp)\otimes \cC \rightarrow \Sh(M;\cC)$ is identified with $\Sh_Z(M;\cC)$. Equivalently, we have $i_{Z}^\cC\simeq i_{Z}^\Sp\otimes \cC$ and
\[\Sh_Z(M;\cC)\simeq \Sh_Z(M;\Sp)\otimes \cC . \] 
\end{Prop}
\begin{proof}By \autoref{lemma: Omega-lens2}, we have that $(i_{Z}^\Sp)^l$ can be characterized as a presentable left Bousfield localization that is local with respect to morphisms $\bS_{ \{g_0<0\}}\rightarrow \bS_{\{g_1<0\}}$ such that $g=g_t$ defines a $- \dot{T}^*M\setminus Z$-lense. Let us denote by $W_{Z}$ the set of morphisms (since $\Omega$-lenses form a set).

We identify $\Sh(M;\cC)=\Fun^L(\Sh(M;\Sp),\cC)$ and $\Sh_Z(M;\Sp)\otimes \cC=\Fun^L(\Sh_{-Z}(M;\Sp),\cC)$ (i.e. corresponding cosheaf categories) by dualizability, we have that $\Sh_Z(M;\Sp)\otimes \cC$ consists exactly colimit preserving functors $\Sh(M;\Sp)\rightarrow \cC$ that are local with respect to $W_{Z}$. However, those $W_{Z}$-local colimit preserving functors form exactly $\Sh_Z(M;\cC)$ in $\Sh(M;\cC)=\Fun^L(\Sh(M;\Sp),\cC)$ by the definition of $\Sh_Z(M;\cC)$. The result then follows from the above discussion and \cite[Proposition 5.5.4.2]{HTT}.\end{proof}

\begin{Coro}If $\cC$ is presentable stable, then we have $\SS_{\cC}(F)$ exists. Consequently, all microsupport estimation results in \cite{KS90} are correct for $\cC$-valued sheaves.   
\end{Coro}
\begin{proof}Under the condition that $\cC$ is presentable stable, we identify $\Sh_Z(M;\cC)\simeq \Sh_Z(M;\Sp)\otimes \cC$. Therefore, if the sheaf $F$ is regular on $\Omega_1$ and $\Omega_2$, $F$ is regular on $\Omega_1\cup \Omega_2$ because this is true when $\cC=\Sp$ by pointwise verification. Then we just take $\dot{SS}_\cC(F)$ as the intersection of closed sets that $F$ is not regular. 

For the statement for estimation, it also follows from corresponding $\Sp$-valued estimation due to comparison of $\cC$-valued $6$-functors and $\Sp$-valued $6$-functors in \cite{6functor-infinity}.
\end{proof}

\begin{RMK}After the first version of the article, we learn that Tamarkin prove the existence of $\SS_{\cC}(F)$ only when $\cC$ is (unstably!) presentable in \cite{Tamarkin_notes}. 

The idea is the following: Recall the set of morphism $W_Z$ described in \autoref{defintion: ad-hoc def of SS} and we denote the saturation $\overline{W_Z}$ (i.e, the minimal strongly saturation class of morphisms, in the sense of \cite[Definition 5.5.4.5]{HTT}, that contains $W_Z$). Then using a partition of unity argument, one can prove that for any set of closed sets $Z_i$ and $Z\coloneqq \cap_i Z_i$, we have $\overline{W_Z}=\overline{\cup_i W_{Z_i}}$, which implies that the regular open sets of $F$ are closed under union.
\end{RMK}

\begin{RMK}We should be careful when discussing monoidal structures on $\cC$ and related $\cC$-linear dualizability. For example, one can assume $\cC$ is (dualizable stable) locally rigid, see \cite{Ramzi_locally_rigid_cat}. In this note, we only (need to) discuss the monoidal structure on $\Sp$, which is known to be rigid.
\end{RMK}

Now, we consider the following category introduced by Tamarkin. For a dualizable stable category $\cC$, we set
\[\sT_\cC\coloneqq \Sh(\bR;\cC)/\Sh_{\bR\times (-\infty,0]}(\bR;\cC).\]

Using the bi-fiber sequence $\sT_\Sp\rightarrow \Sh(\bR;\Sp)\rightarrow \Sh_{\bR\times (-\infty,0]}(\bR;\Sp)$ in $\Catdual$, we have that $\sT_\cC\simeq \sT_\Sp\otimes \cC$, and for a smooth manifold $M$ we have 
\begin{equation}\label{equations: def of Tam}
    \Sh(M;\Sp)\otimes \sT_\cC \simeq \Sh(M; \sT_\cC) \simeq \Sh(M \times \bR;\cC)/\Sh_{T^*M\times \bR\times (-\infty,0]}(M \times \bR;\cC).
\end{equation}

To avoid certain confusion, we denote the three equivalent categories by
\[\sT(T^*M;\cC).\]

We refer to \cite{Hochschild-Kuo-Shende-Zhang} for both the motivation of this notation and the proof of \eqref{equations: def of Tam} (for the case $\cC = \Sp$, and the general case follows from \autoref{defintion: ad-hoc def of SS}.)

Therefore, for $[F]\in \sT(T^*M;\cC)$, we can discuss the positive part of microsupport, say
\[SS_+([F])\coloneqq SS(F)\cap \{\tau>0\}\]
is a well-defined closed conic subset of $T^*M\times \bR\times (0,\infty)$.

The construction is introduced to study non-conic subsets $Z\subset T^*M$. Precisely, for a subset $Z\subset T^*M$, we define its cone as
\[\widehat{Z}\coloneqq \{(q,p,t,\tau)\in T^*M\times \bR\times (0,\infty): (q,p/\tau)\in Z\}.\]
The conic set $\widehat{Z}$ is closed in $T^*M\times \bR\times (0,\infty)$ if $Z$ is closed in $T^*M$, and in this case we denote $\sT_{Z}(T^*M;\cC)$ the full subcategory spanned by those $[F]$ such that $SS_+([F])\subset \widehat{Z}$.

\begin{Prop}\label{prop: kunneth}For a conic closed set $Z\subset T^*M$, we have
\[\sT_{Z}(T^*M;\cC) \simeq \Sh_{Z}(M;\Sp)\otimes \sT_\cC \simeq \Sh_{Z}(M;\cC)\otimes \sT_\Sp.\]  
\end{Prop}

\begin{proof}If $Z$ is already conic, then we have $\widehat{Z}= Z\times \bR\times (0,\infty). $ Therefore, we have
\[\sT_{Z}(T^*M;\cC)=\Sh_{Z\times \bR\times (0,\infty)}(M\times\bR; \cC).\]

Then we can apply the Künneth formula for the category of sheaves with general microsupport condition, see \cite{Hochschild-Kuo-Shende-Zhang} or \cite{Cutoff-Zhang}. In loc. cit., the Kunneth formula is proved for compactly generated rigid symmetric monoidal $\cC$. In particular, for $\cC=\Sp$. As a result, we have
\[\sT_{Z}(T^*M;\Sp)=\Sh_{Z\times \bR\times (0,\infty)}(M\times\bR; \Sp)=\Sh_{Z }(M; \Sp)\otimes \Sh_{\bR\times (0,\infty)}(\bR; \Sp)=\Sh_{Z }(M; \Sp)\otimes \sT_\Sp.\]

Then for more general coefficients $\cC$, the result is a formal consequence of \autoref{defintion: ad-hoc def of SS}.
\end{proof}

\section{Fourier-Sato-Tamarkin transform} 
In this section, we only consider $\cC=\Sp$. This restriction does not affect generality in our discussion due to \autoref{defintion: ad-hoc def of SS}.

The Fourier-Sato transform was first introduced by Sato in \cite{SKK73}. We refer to \cite[Section 3.7]{KS90} for more relevant discussion. The Fourier-Sato transform gives an equivalence between $\bR_{> 0}$-equivariant sheaves on $V$ and $V^\vee$ for a real vector space $V$. To adapt to various non-equivariant situations, one can consider some variants of the Fourier-Sato transform. We refer to \cite{FTDAgnolo,radonHonghao} for more relevant discussion on their definition and the comparison between them.

In \cite{tamarkin2013}, Tamarkin introduces a variant of the Fourier-Sato transform that induces an equivalence for sheaves that are not necessarily $\bR_{> 0}$-equivariant. We call it the Fourier-Sato-Tamarkin transform, and we explain its definition now.

We naturally identify both $T^*V$ and $T^*V^\vee$ with $V\times V^\vee$. Let $Leg(V)=\lbrace (z,\zeta,t,s): t-s+\langle z,\zeta\rangle\geq 0 \rbrace \subset V\times V^\vee\times \bR^2$. We consider
\begin{align*}
    &{\mathbb{S}}_{Leg(V)}\in \Sh(V\times V^\vee\times \bR^2;\Sp),\\
    &p_{V}: V\times V^\vee \times \bR_t\times \bR_s\rightarrow  V\times \bR_s,\\
    &p_{V^\vee}: V\times V^\vee \times \bR_t\times \bR_s\rightarrow  V^\vee\times \bR_t.
\end{align*}

\begin{Def}The Fourier-Sato-Tamarkin transform is defined as the functor
\begin{align*}
    &{\FST}:  \sT(T^*V;\Sp)  \rightarrow  \sT(T^*V^\vee;\Sp) ,\\
    &{\FST}(F)\coloneqq  p_{V^\vee!} (p_{V}^* F \otimes_{\Sp} {\bS}_{Leg(V)})[\dim V].
\end{align*}

\end{Def}
It is proved in \cite[Theorem 3.5]{tamarkin2013} that the Fourier-Sato-Tamarkin transform ${\FST}$ is an equivalence of categories.

\begin{Thm}[{\cite[Theorem 3.6]{tamarkin2013}}]\label{thm: Fourier-sato-tamarkin}The Fourier-Sato-Tamarkin transform ${\FST}$ induces the following equivalence of categories: For a closed set $X\subset V^\vee$, we have
\[\sT_{{V\times X}}(T^*V;\Sp)  \simeq    \sT_{{X\times V}}(T^*V^\vee;\Sp).\]
\end{Thm}

This yields the following result
\begin{Prop}\label{prop: kunneth2}Let $V$ be a finite dimensional real vector space and $X\subset V^\vee$ be a closed set. We have 
\[\sT_{{V\times X}}(T^*V;\Sp)   \simeq  \Sh(X;\sT_\Sp).\]    
\end{Prop}
\begin{proof}By \autoref{thm: Fourier-sato-tamarkin}, we have $\sT_{{V\times X}}(T^*V;\Sp)  \simeq    \sT_{{X\times V}}(T^*V^\vee;\Sp)$. We notice that $V$ is a conic closed set in $V=(V^\vee)^\vee$. Then $\sT_{{X\times V}}(T^*V^\vee;\Sp)\simeq \Sh_{X\times V}( V^\vee;\Sp)\otimes \sT_\Sp$ by \autoref{prop: kunneth}.

Lastly, recall that $\Sh_{X\times V}( V^\vee;\Sp)$ consists of sheaves $H\in \Sh(V^\vee;\Sp)$ with the usual microsupport bound $\SS(H)\subset X\times V$. We notice that $\pi(\SS(H))=\supp(H)$ for the cotangent projection $\pi$. Henceforth, we have $\SS(H)\subset X\times V$ if and only if $H$ is supported in $X$. In particular, it shows that $\Sh_{X\times V}( V^\vee;\Sp)\simeq \Sh( X;\Sp)$.
\end{proof}

\begin{RMK}For $V=\bR$ and $X=(-\infty,0]$. As a consequence of \autoref{prop: kunneth2}, we can pass to the quotient of $\Sh(\bR;\sT_\Sp)$ by $  \sT_{{\bR\times (-\infty,0]}}(T^*V;\Sp)  \simeq  \Sh((-\infty,0];\sT_\Sp)$ to see that
\[\sT_\Sp\otimes \sT_\Sp\simeq \Sh((0,\infty);\sT_\Sp).\]

This is \cite[Proposition 6.1]{WildBetti}, and the proof here is a precise form of \cite[Warning 6.2]{WildBetti} therein.
\end{RMK}

\section{Localizing invariants}
Let $\cE$ be an accessible stable category. Recall that a functor $F:\Catdual\to\cE$ is a (continuous) localizing invariant if the following conditions hold:

\begin{enumerate}[label=(\roman*),ref=(\roman*)]
\item $F(0)=0;$ 

\item for any bi-fiber sequence of the form
\[\cA \rightarrow \cB \rightarrow\cC\]
in $\Catdual$, the sequence
\[F(\cA) \rightarrow F(\cB) \rightarrow F(\cC)\]
is a fiber sequence in $\cE$.
\end{enumerate}

Roughly speaking, the results of \cite{Efimov-K-theory} assert that a localizing invariant is determined by its value on compactly generated stable categories, and all localizing invariants come in this way. 

\begin{RMK}In general, we should be careful about the accessibility of localizing invariants. Here, we will only discuss finitary localizing invariants, i.e. those commute with $\omega$-filtered colimits.    
\end{RMK}

Among all finitary localizing invariants, there exists a universal one, which was originally studied by \cite{universal-hihger-K} on idempotent complete stable (small) categories, or equivalently compactly generated stable categories, and was later extended to dualizable stable categories by \cite{Efimov-K-theory}. 

More precisely, there exists an $\omega$-accessible stable category $\Motloc$ that is called the category of non-commutative motives and a universal finitary localizing invariant $\Uloc:\Catdual\rightarrow \Motloc$ which is initial among all finitary localizing invariants:
\[  \Fun^L(\Motloc,\cE )\simeq \Fun_{\textnormal{loc},\omega}(\Catdual,\cE ),\quad G\mapsto F=G\circ\Uloc .\]

Therefore, many properties of $\Uloc$ are automatically shared by all finitary localizing invariants.

The following standard observation follows directly from the definition.
\begin{Lemma}\label{lemma: trivial lemma}For a localizing invariant $F: \Catdual \rightarrow \cE$ and a dualizable stable category $\cC$, we have $F(\cC\otimes -):\Catdual \rightarrow \cE$ is a localizing invariant.    
\end{Lemma}

Now, we will discuss localizing invariants of $\Sh_{V\times X}(V;{\cC})\simeq \Sh_{V\times X}(V;{\Sp})\otimes \cC$. Then by \autoref{lemma: trivial lemma}, we may assume $\cC=\Sp$ in the following proofs.

To start with, we present a proof of the 1-dimension \autoref{Theorem: Efimov with support} based on \autoref{Theorem: Efimov no support}. The idea is to reverse the process of the original proof presented in \cite{Efimov2022ICM}. 

\begin{Prop}\label{prop: 1.2 from 1.1}We have $\Uloc (\Sh_{\bR\times [0,\infty)}(\bR;\cC)) \simeq 0$ and $\Uloc ({\sT_\cC}) \simeq \Omega\Uloc ({\cC})$.    
\end{Prop}
\begin{proof}One can check that we have the following Cartesian square of dualizable stable categories. See also \cite[Theorem 6.16]{APT-Tatsuki-Zhang} for a detailed proof of its generalization.
\[\begin{tikzcd}
\Sh(\bR;\Sp) \arrow[d] \arrow[r]                 & {\Sh_{\bR\times [0,\infty)}(\bR;\Sp)} \arrow[d] \\
{\Sh_{\bR\times (-\infty,0]}(\bR;\Sp)} \arrow[r] & \Sh_{\bR\times \{0\}}(\bR;\Sp) .                
\end{tikzcd}\]

The map $x\mapsto -x$ identifies ${\Sh_{\bR\times [0,\infty)}(\bR;\Sp)}\simeq {\Sh_{\bR\times (-\infty,0]}(\bR;\Sp)} $. Therefore, by \cite[Proposition 4.11]{Efimov-K-theory}, we have the fiber sequence in $\Motloc$:
\[  \Uloc (\Sh_{\bR\times [0,\infty)}(\bR;\Sp))^{\oplus 2}\rightarrow               \Uloc ( \Sh_{\bR\times \{0\}}(\bR;\Sp))\rightarrow \Sigma  \Uloc (\Sh(\bR;\Sp)). \]

One can directly check that the second morphism is induced by the loop-suspension adjunction of $\Uloc ( \Sp)$, which yields an equivalence. Then we have $\Uloc (\Sh_{\bR\times [0,\infty)}(\bR;\Sp)) $ vanishes.

Next, we consider the bi-fiber sequence
\[\sT_\Sp\rightarrow \Sh_{\bR\times [0,\infty)}(\bR;\Sp)\rightarrow \Sp.\]

Then the second statement follows directly from the fact that $\Uloc$ is a localizing invariant and the first equivalence.
\end{proof}

Now, we can state our theorems.

\begin{Thm}\label{Theorem: main thm}Let $V$ be a finite dimensional real vector space and let $X\subset V^\vee$ be a closed set, and $\cC$ be a dualizable stable category.
\begin{enumerate}[fullwidth]
        \item We have the equivalence:\[\Uloc (\sT_{{V\times X}}(T^*V;\cC)) \simeq \Gamma_c(X; \Uloc({\sT_\cC}) )\simeq \Omega\Gamma_c(X; \Uloc({\cC}) );\]
        \item If $X$ is in addition conic, then we have
        \[\Uloc (\Sh_{V\times X}(V;{\cC})) \simeq \Gamma_c(X; \Uloc({\cC}) ).\]
\end{enumerate}  
In particular, the same statements hold for any finitary localizing invariant in place of $\Uloc$.
\end{Thm}

\begin{proof}

\begin{enumerate}[fullwidth]
\item This is a direct corollary of \autoref{thm: Fourier-sato-tamarkin},  \autoref{Theorem: Efimov no support} and \autoref{prop: 1.2 from 1.1}.
    
\item By \autoref{prop: kunneth}, when $X$ is conic, we have
\[\sT_{{V\times X}}(T^*V;\cC) \simeq \Sh_{V\times X}(V;\cC)\otimes{\sT_\Sp}.\]

By \autoref{lemma: trivial lemma} and (1), we have for any localizing invariants $F$ that $F(\sT_\Sp)\simeq \Omega F(\Sp)$. Taking $F=\Uloc( \Sh_{V\times X}(V;\cC) \otimes -)$, we obtain
\begin{align*}
    \Uloc(\Sh_{V\times X}(V;\cC))\simeq &\Sigma\Uloc(\Sh_{V\times X}(V;\cC)\otimes \sT_\Sp) \\
    \simeq &\Sigma\Uloc(\sT_{V\times X}(T^*V;\cC)) \\
    \simeq &\Sigma\Omega \Gamma_c(X; \Uloc({  \cC}) )
    \\
    \simeq &\Gamma_c(X; \Uloc({  \cC}) ).
\end{align*}

Then we conclude the proof.\qedhere
\end{enumerate}
\end{proof}

Let $X \subset V^\vee$ be closed, and set $U = V^\vee \setminus X$, we can consider the quotient category 
\[\sT(V\times U;\cC)\coloneqq \sT(T^*V;\Sp)/\sT_{{V\times X}}(T^*V;\Sp);\]
and, if $X$ is conic and then $U$ is also conic, we can consider
\[\Sh(V,V\times U;\cC)\coloneqq\Sh(V;\cC)/\Sh_{{V\times X}}(V;\cC)\]
as in \cite[Definition 6.1.1]{KS90}.

We then obtain the following consequence:
\begin{Coro}Let $V$ be a finite dimensional real vector space and a closed set $X\subset V^\vee$ with $U=V^\vee\setminus X$.
\begin{enumerate}[fullwidth]
        \item We have the equivalence:\[\Uloc(\sT(V\times U;\cC)) \simeq \Omega\Gamma_c(U; \Uloc({\cC}) );\]
        \item If $X$ is in addition conic, then we have
        \[\Uloc (\Sh(V,V\times U;\cC)) \simeq \Gamma_c(U; \Uloc({\cC}) ).\]
\end{enumerate}  
In particular, the same statements hold for any finitary localizing invariant in place of $\Uloc$.
\end{Coro}
\begin{proof}We prove (2); the argument for (1) is analogous.

By \autoref{lemma: Omega-lens2}, we know the quotient functor $\Sh(V;\cC)\rightarrow \Sh_{{V\times X}}(V;\cC)$ is strongly continuous. This yields a fiber sequence in $\Catdual$
\[\Sh(V,V\times U;\cC) \rightarrow \Sh(V;\cC) \rightarrow \Sh_{{V\times X}}(V;\cC).\]
Applying $\Uloc$ to the sequence, then the result follows from the fiber sequence
\[\Gamma_c(U; \Uloc({\cC}) )\rightarrow\Gamma_c(V^\vee; \Uloc({\cC}) )\rightarrow\Gamma_c(X; \Uloc({\cC}) ).\qedhere\]
\end{proof}

\subsection{Further questions}\label{subsection: questions}

At the end of this note, we mention that for a general cotangent bundle $T^*M$ and an open set $D\subset T^*M$, it is known that the Tamarkin category
\[\sT(D;\cC) = \sT(T^*M;\cC)/\sT_{T^*M\setminus D}(T^*M;\cC)\]
has important symplectic geometric information of the open symplectic manifold $D$, and the filtered Fukaya category $\mathcal{F}^{fil}(D)$ is a full subcategory of $\sT(D;\cC)$. 

For example, the Hochschild homology of $\sT(D;\cC)$ was studied in \cite{CyclicZHANG,Hochschild-Kuo-Shende-Zhang}. It is proven that corresponding $\sT_{\Mod_\bK}$-linear Hochschild (co)homology is equivalent to the filtered symplectic cohomology of $U$ (when $D$ has good contact boundary). 

One may be naturally interested in the computation of 
\[\Uloc(\sT(D;\cC))\]
or any specific finitary localizing invariants instead of $\Uloc$. Here, our result is the first attempt for the question for $D=V\times U$ for a real vector space $V$ and an open set $U\subset V^\vee$.

\bibliographystyle{bingyu}
\clearpage
\phantomsection
\bibliography{bibtex}

\end{document}